\documentclass[12pt]{article}
 
\usepackage[a4paper, width=170mm, top=35mm, bottom=35mm, bindingoffset=6mm]{geometry}

\usepackage{amssymb} 
\usepackage{amsmath} 
\usepackage{amssymb,bm,amsmath,amssymb,amsthm}
\usepackage[]{inputenc}
\usepackage[small,nohug,heads=vee]{diagrams}
\usepackage{tikz-cd}

\diagramstyle[labelstyle=\scriptstyle]

\usepackage{authblk}

\usepackage[all]{xy}
  \SelectTips{cm}{10}     
  \everyxy={<2.5em,0em>:} 

\usepackage{hyperref}

\DeclareMathOperator{\Gal}{Gal}
\DeclareMathOperator{\End}{End}
\DeclareMathOperator{\sep}{sep}
\DeclareMathOperator{\Ind}{Ind}

\DeclareMathOperator{\reg}{reg}
\DeclareMathOperator{\GL}{Gl}

\newtheorem{ex}{Example}

\newtheorem{defin}{Definition}
\newtheorem{theorem}{Theorem}
\newcommand{\overbar}[1]{\mkern 1.5mu\overline{\mkern-1.5mu#1\mkern-1.5mu}\mkern 1.5mu}

\newcommand{\Q}{\mathbb Q}

\begin{document}
\title{Notes on Explicit Constructions of Arithmetically Equivalent Global Function Fields via Torsion Points On Drinfeld Modules.}
\author{Pavel Solomatin \\
\texttt{pavelsolomatin179@gmail.com}
}

\date{ Leiden, 2021.}
\maketitle  

\begin{abstract}
In this short note we provide a few examples of non-isomorphic arithmetically equivalent global function fields. These examples are obtained via well-known technique of adjoining the torsion points of various Drinfeld Modules to realise the $\GL_n(\mathbb F_q)$ as a Galois group of extensions of global function fields. Furthermore we afford the code of the Magma scripts to verify the results and construct more examples in similar fashion.  
\end{abstract}

\textbf{\\ \\ \\ \\ \\ \\ \\  \\ \\ \\ \\ \\  \\ Acknowledgements:} 
I would like to thank professor David Zywina for providing relevant references and answering my questions.

\newpage

\section{Introduction}
\subsection{Arithmetically Equivalent Number Fields}
Given a number field $K$, i.e. a finite separable extension of the field of rational numbers $\Q$, one can study the following informal question: \emph{What kind of information the splitting type of the rational primes $(p)$ in the ring of integers $\mathcal O_K$ captures about the field $K$?} This question leads naturally to the beautiful theory of \emph{arithmetically equivalent number fields}. In essence the theory allows one to construct pairs of non-isomorphic number fields $K$,~$L$ which share a lot of arithmetic invariants. We start from a quick remainder about the basic ingredients of the construction. 

We say that  two number fields $K$ and $L$ \emph{split equivalently} if for every except finitely many prime numbers $p \in \mathbb Z$ there exists a bijection $\phi_p$ from the set of primes in $\mathcal O_K$ lying above $p$ to the set of those primes in $\mathcal O_L$. If for every except finitely many $p$ there exists such a bijection which preserves the degrees of the ideals then we say that $K$ and $L$ are \emph{arithmetically equivalent}. Also let $N$ denote the common Galois closure of $K$ and $L$ over $\mathbb Q$ and let $G= \Gal(N / \mathbb Q) $, $H = \Gal(N / K) $, $H' = \Gal(N / L) $. We will call a triple $(G,H, H')$ a \emph{Gassmann triple} if the following induced representations are isomorphic: $ \Ind^{G}_{H} (1_{H}) \simeq \Ind^{G}_{H'} (1_{H'})$, where $1_H$(and $1_{H'}$) means trivial representation of $H$(of $H'$ respectively). Then the following famous Theorem due to Perlis \cite{Perl} asserts that:
\begin{theorem}[Perlis]\label{MainPerl}
The following statements are equivalent: 
\begin{enumerate}
        \item  $K$ and $L$ split equivalently; 
        \item  $K$ and $L$ are arithmetically equivalent; 
        \item 	$(G,H,H')$ form a Gassmann triple;
        \item $\zeta_K(s) = \zeta_L(s)$.  		
\end{enumerate}
Moreover, if the above conditions hold then $K$ and $L$ share the following invariants: the degree, the Discriminants $D_K = D_{L}$, group of units $\mathcal {O}_{K}^{\times} \simeq \mathcal {O}_{L}^{\times}$, product of regulator and class-number $\reg(\mathcal O_K)h_K = \reg(\mathcal O_L)h_L$ and etc.
\end{theorem}
   
According to the Galois theory $K$ and $L$ are not isomorphic if and only if $H$ and $H'$ are not conjugate inside $G$. In the later case we say that the Gassmann triple $(G, H, H')$ is non-trivial. We will also say that a Gassmann triple $(G, H, H')$ has index $n$, where $n = \frac{|G|}{|H|} = \frac{|G|}{|H'|}$. In other words the index $n$ is the degree of the number fileds $K$, $L$. The above mentioned Theorem provides us with the powerful method to construct non-isomorphic arithmetically equivalent number fields. Namely, one first finds a non-trivial Gassmann triple $(G, H, H')$ and then solves explicitly the inverse Galois problem to realise $G$ as the Galois group $\Gal(N : \mathbb Q)$ of a field extension $N$ over $\mathbb Q$. The classification of non-trivial Gassmann triples of small degrees is well-known. For instance, all the possible triples of degree not exceeding 15 were classified in \cite{Bart2}. Below we provide some classical instances:

\begin{ex}\label{ExTransp}
Fix a prime number $p>2$. Let $G$ be $\GL_{2}(\mathbb F_p)$ and let $H = \left \{ \begin{bmatrix} 1 & * \\ 0 & * \end{bmatrix} \in G \right \}$ and $H' = \left \{ \begin{bmatrix} * & * \\ 0 & 1 \end{bmatrix} \in G \right \}$. Then the triple $(G,H,H')$ is a non-trivial Gassmann triple.
\end{ex}

In similar fashion one constructs Gassmann triples with $G = \GL_n(\mathbb F_q)$ and some of its quotients, see \cite{Bart2}:
\begin{ex}\label{Example2}
Let $S$ be a subgroup in $\mathbb F_q^{\times}$ of index $s$ and let $G = \GL_n(\mathbb F_q)/S$, $n \ge 2$. By construction $\GL_n(\mathbb F_q)$ is endowed with an action on the vector space $V = {\mathbb F_q}^{n} $ and also on its dual vector space $V^{\star}$. One defines $H$, $H' \subset G$  as stabilizers of $(V - \{ 0 \})/S$ and $(V^{\star} - \{ 0 \})/S$ respectively. Then if $s \ge 2$ or $n \ge 3$ then $(G, H, H')$ is a non-trivial Gassmann triple of index $\frac{s(q^n - 1)}{q-1}$.
\end{ex} 

This shows that solving the inverse Galois problem for groups of $\GL_n(\mathbb F_q)$ type allows one to produce arithmetically equivalent number fields. For $n=2$ and $q = p$ been a prime number this can be achieved by using $p$-torsion points of elliptic curves defined over $\mathbb Q$. Roughly speaking the Serre's open image Theorem asserts that for an elliptic curve $E$ without complex multiplication and a random prime number $p$ the Galois group of the normal closure of the field extension generated by the coordinates of $p$-torsions $E[p]$ is isomorphic to $\GL_2(\mathbb F_p)$. This method was used for example in \cite{Bart3} to construct families of arithmetically equivalent number fields with non-isomorphic class-groups. 

Motivated by the number field case treated above one can ask what is the function field analogue of the arithmetical equivalence. Despite the fact that this question has been already addressed in some research projects, there are not so many known explicit examples. So with this text we are trying to improve the situation. For now we focus our attention on the case of arithmetically equivalent global function fields.

\subsection{Arithmetically Equivalent Global Function Fields}
Let $q=p^m$, $p$ be a prime number. For the purpose of this note we use the following definition: a \emph{global function field} is a finite, separable extension of the field $F = \mathbb F_q(T)$. In particular \emph{all the morphisms between global function fields should fix the base field $F$}. As a consequence of this requirement there exist non-isomorphic function fields which are isomorphic as abstract fields. Recall that the polynomial ring $A=\mathbb F_q[T]$ is an analogue of $\mathbb Z$ with prime numbers replaced by monic irreducible polynomials and that prime ideals in $\mathbb F_q[T]$ are in one-to-one correspondence with monic irreducible polynomials in $\mathbb F_q[T]$. Given a finite separable geometric extension $K$ of $F$ we consider the integral closure $\mathcal O_K$ of $\mathbb F_q[T]$ in $K$. As in the number field case ideals of $\mathcal O_F = \mathbb F_q[T]$ split in $\mathcal O_K$ and one can address the same question regarding the arithmetic properties of $K$ captured by the splitting of ideals in $\mathcal O_K$. In the given function field settings the notions of arithmetical, split and Gassmann equivalences are transferred without any changes and an analogue of Theorem \ref{MainPerl} holds with only one exception. Namely, the first three conditions are still equivalent and imply the equality of zeta-functions, but the later condition is not enough to deduce one of the other three equivalent statements, see \cite{Gunt1} and \cite{ArEq}. 

This result implies that again the realisation of $\GL_n(\mathbb F_{q'})$ as a Galois group\footnote{with $q'$ not necessarily equals to $q$.} of an extension of $\mathbb F_q(x)$ allows us to produce pairs of arithmetically equivalent function fields. But in contrast to the number filed case there are two different ways to do so: either using \emph{torsion points of elliptic curves defined over $\mathbb F_q(T)$} or using \emph{torsion points of Drinfeld modules}. The first method was applied in \cite{ArEq} to construct arithmetically equivalent pairs coming from $\GL_2(\mathbb F_l)$ with $l \ne p$. In these notes we show how to apply the second method to construct pairs coming from $\GL_n(\mathbb F_{p^{m}})$. In the next section we briefly recall the main definitions with regard to Drinfeld modules.

\section{Preliminaries: Galois Extensions Obtained by Torsion Points on Drinfeld Modules}
\subsection{Required Definitions}
General definition of Drinfeld modules associated to an arbitrary global function field $K$ is rather complicated, but in this chapter we provide a more elementary one which is sufficient for our purposes. All the definitions and examples given below are quite standard and therefore stated without exact references. For more detailed explanations an interested reader could consult last chapters of \cite{Rosen} as well as more advanced materials, for instance \cite{Goss} and \cite{thakur2004function}. 

Let $k$ be any field of characteristic $p$ endowed with a homomorphism $\phi:~ A = \mathbb F_q[T]~\to~k$. One considers \emph{the ring of twisted polynomials} $k\left<\tau\right>$ where the multiplication law is defined by the rule $\tau a = a^q \tau $ for $a \in k$. Let $\rho$ be an $\mathbb F_q$-algebra homomorphism from $\mathbb F_q[T]$ to $k\left<\tau\right>$. This homomorphism is defined by the image of $T$ in $k\left<\tau\right>$ which we denote by: $$\rho_{T} = \sum_{i=0}^{N} a_i \tau^{i},$$
with $a_i \in k$, $a_N \neq 0$.

\begin{defin}
A Drinfeld module for $\mathbb F_q[T]$ over $k$ is an $\mathbb F_q$-algebra homomorphism $\rho$ from $A$ to $k\left<\tau \right>$ such that $N>0$ and $a_0 = \phi(T)$. The number $N$ is called rank of $\rho$ and the kernel of $\phi$ in $A$ is called characteristic of $\rho$. If $\phi$ is injective we say that $\rho$ has general characteristic.
\end{defin}

The main working example for us is $k = F = \mathbb F_q(T)$ and $\phi$ be the standard embedding $A \to \mathbb F_q(T)$. In particular this implies that $a_0 = T$ and that $\rho$ is of general characterstic. Associating $\tau$ with the Frobenius map $x \to x^{q}$ each Drinfeld module $\rho$ over $F$ endows $A$ with a non-trivial action on the algebraic closure $\overbar{F}$ of $F$. Namely for an element $u \in A$ the image $\rho_u = \sum a_i(u)\tau^i$ acts on $x \in \overbar{F}$ as $ x \to \sum a_i(u) x^{p^{i}}$. An important observation is that the multiplication law in $k \left< \tau \right>$ is agreed with the composition of such operators.

\begin{ex}
The most famous example of Drinfeld modules is the so-called Carlitz module which is defined by $\rho_T = \tau + T$. For instance the corresponding to $u = T^2 + 1$ operator is given by $\rho_{T^2 + 1} = (\tau + T)(\tau + T) + 1 = \tau^2 + (T^q + T)\tau + T^2 + 1$ which sends $x \in {F}$ to $\rho_{T^2 + 1}(x) = x^{p^2} + (T^q + T)x^p + (T^2 + 1)x$.
\end{ex}

\begin{defin}
Given two Drinfeld modules $\rho$, $\rho'$ over $k$ one defines a $k$-morphism from $\rho$ to $\rho'$ as an element $u \in k \left<\tau\right>$ such that $\rho u = u \rho'$. The centralizer of the image $\rho(A)$ in $\overbar{K}\left<  \tau \right>$ is called the endomorphism ring which we denote by $$\End_{\overbar{k}}(\rho) = \{ u \in \overbar{k}\left< \tau \right> | \rho u = u \rho \}.$$
\end{defin}
 
Note that  $\End_{\overbar{k}}(\rho)$ contains $\rho(A)$ and is a free $\rho(A)$-module of rank at most $r^2$. The following notions play the crucial role in the main construction we are going to use: 

\begin{defin}
We say that $\rho$ admits complex multiplication if there exists $a \in \End_{\overbar{k}}(\rho)$, with $a \not \in \rho(A)$. 
\end{defin}

\begin{defin}
Given $a \in A$ one defines the set of $a$-torsion points of $\rho$ as $$\rho [a] = \left \{ x \in F^{\sep} \ | \ \rho_a(x) =  0  \right \}.$$ 
\end{defin}
It is not difficult to see that as $A$-module one has an isomorphism: $$\rho [a] \simeq (A/aA)^{r},$$ and that the Galois group of $F^{\sep}$ is acting on $\rho [a]$ which induces the representation: 
\begin{equation}\label{repres}
\mathcal G_F = \Gal( F^{\sep} : F ) \to \GL_r ( A/aA ).
\end{equation}

\subsection{The Open Image Theorem and an Algorithm}
Combing all such representations for various $a \in \mathbb F_q[T]$ one gets an Adelic representation $\mathcal G_F \to \GL_r( \hat{A} )$. The key fact is the following Theorem due to R.Pink and E. R\"utsche which is an analogue of the \emph{Serre's open image Theorem}:

\begin{theorem}
For a Drinfeld module $\rho$ such that $\End_{\overbar{K}}(\rho) = \rho(A)$ the image of $\mathcal G_F$ in $Gl_r( \hat{A} )$ has finite index. In other words, for almost all except finitely many $a \in A$ the map from \ref{repres} is surjective.
\end{theorem}  
\begin{proof}
See Theorem 0.1 from \cite {PinkAdelic}.
\end{proof}

This means that for a Drinfeld module $\rho$ without complex multiplication and for $a$ been a prime ideal of $A$ of degree $m$ the Galois group of the normal closure of the field obtained by the adjoining the $[a]$-torsions points to $\mathbb F_q(T)$ is isomorphic to $\GL_r(\mathbb F_{q^m})$ except for finitely many choices of $a$. This leads to the following probabilistic algorithm to construct arithmetically equivalent function fields: 

\begin{enumerate}
  \item Pick a Drinfeld module $\rho$ with $\End_{\overbar{K}}(\rho) = \rho(A)$ and a prime ideal $a$ of $A$ of degree $m$;
  \item Calculate the Galois group $G$ of the normal closure of the filed $F(\rho[a])$;
  \item If $G \simeq \GL_r(\mathbb F_{q^m})$ find two subgroups $H$, $H'$ of G which form a non-trivial Gassmann triple $(G, H, H')$ from \ref{ExTransp}. Otherwise pick another $a$; 
  \item The fixed fields $F(\rho[a])^{H}$ and $F(\rho[a])^{H}$ are non-isomorphic arithmetically equivalent function fields.
\end{enumerate}

For an arbitrary Drinfeld module $\rho$ the question of describing the endomorphism ring $\End_{\overbar{k}}(\rho)$ explicitly is not elementary, but can be solved algorithmically as it was shown in \cite{EndomorphDrin}. In contrast, this problem is significantly simpler for  the case of rank-two Drinfeld modules. In this setup one defines the $j$-invariant of $\rho$ which identifies the set of isomorphism classes of Drinfeld modules with the points of the affine line $\mathbb A_1(K)$. Similar to the case of elliptic curves the values of $j(\rho)$ for $\rho$ with complex multiplication should be integral. Therefore a Drinfeld module of rank two with $j \not \in A$ will not have complex multiplication. In addition to that D.Zywina \cite{Zywina2011DrinfeldMW}  constructed an explicit example of a Drinfeld module of rank two with the surjective Adelic representation. A similar example for rank 3 modules has been recently given in \cite{CHEN2020}. Those instances could be used to exclude the "probabilistic" part of the algorithm.
\section{Examples and the Code}
\subsection{Example with the Galois Group $\GL_2(\mathbb F_3)$ over $\mathbb F_3(T)$}
Let us first consider an example of a realisation of $\GL_2(\mathbb F_3)$ as a Galois group of an extension of $\mathbb F_3(T)$. According to the general theory we have to pick an irreducible polynomial $h(\tau)$ of degree two in the ring of skew-polynomials $k\left<\tau \right>$ and generate an extension by using the $T$-torsion\footnote{or any other linear polynomial.} points of the corresponding Drinfeld module. Let us take for instance $h(\tau)~=~\tau^2 + T\tau + T$ and consider the Drinfeld module which is defined	 by the map $T \to h(\tau)$. The associated $T$-torsion points are elements of the kernel of the operator which maps a polynomial $a \in \mathbb F_3[T]$ to $a^{p^2} + T a^{p} + Ta$. This means that the extension generated by the roots of $a^{8} + Ta^2 + T$ has Galois group of normal closure $N$ isomorphic to $\mathbb \GL_2(\mathbb F_3)$. We can verify this with the following Magma script:
\begin{verbatim}
// Constructing the Drinfeld Module rho(T) = tau^2 + T*tau + T
p := 3;
Fq := GF(p);
k<T> := RationalFunctionField(Fq); 
R<tau> := TwistedPolynomials(k);
rho := tau^2 + T*One(R)*tau + T*One(R);

// Constructing extension by adjoining the T-torsion points of rho
P<y> := PolynomialRing(k);
f<y> := Polynomial(rho)/y;
torsionPolynomial := P!f;
torsionPolynomial;
F<alpha> := FunctionField(torsionPolynomial);
G, r, N := GaloisGroup(F);
G0 := GL(2,p);
IsIsomorphic(G,G0);
\end{verbatim}

The output ensures us that indeed the $T$-torsion points are defined by $y^8 + Ty^2 + T$ and that $G := \Gal( k(\rho[T]) : k ) \simeq \GL_2(\mathbb F_3)$:
\begin{verbatim}
y^8 + T*y^2 + T
true Mapping from: GrpPerm: G to GrpMat: G0
...
\end{verbatim} 

Now we can find the arithmetically equivalent pair from example \ref{ExTransp} as subfileds of $N$. In our example the order of $G$ is $48$ and order of $H$ is $p(p-1) = 6$, so the degree of our pair is 8. To obtain the desired pair we iterate over all index 8 subgroups of $G$ to find the two of them that are Gassmann equivalent and print out the corresponding function fields that are fixed by these subgroups:
\begin{verbatim}
h := Subgroups(G: IndexEqual := 8);
n := #h;
for i in [1..n-1] do
    H_1 := h[i]`subgroup;
    for j in [i+1..n] do
      H_2 := h[j]`subgroup;
      if((PermutationCharacter(G, H_1) eq PermutationCharacter(G, H_2)) and 
          (not IsConjugate(G, H_1, H_2))) then
        K := GaloisSubgroup(N, H_1);
        L := GaloisSubgroup(N, H_2);
        "The group H_1 corresponds to the field extensions: ", K; 
        "The group H_2 corresponds to the field extensions: ", L;    
      end if;
    end for;
end for;
\end{verbatim}

The result shows that the function fields $K, L$ generated by the following polynomials are arithmetically equivalent: 
\begin{equation}
\begin{split}
f(y) \  = \ & y^8 + Ty^2 + T ; \\
g(y) \ = \ & y^8 + T^2y^6 + (2T^6 + T^4 + T^3)y^5 + (2T^5 + 2T^3)y^4 + (2T^9 + 2T^6 + 2T^5)y^3 + \\
     & \ \ \ + (T^{16} + 2T^{13} + T^{10} + 2T^9 + T^8 + 2T^6 + T^5)y^2 + \\
     & \ \ \ + (T^{17} + 2T^{13} + T^{10} + 2T^9 + 2T^8 + 2T^7 + T^6)y + \\
     & \ \ \ + T^{21} + T^{15} + T^{14} + T^{13} + 2T^{12} + T^{11} + 2T^{10} + 2T^9 + 2T^7 + T^6 .
\nonumber
\end{split}
\end{equation}

Since $H_1$, $H_2$ are not conjugate inside $G$ we know that $K$, $L$ are not-isomorphic as function fields, but in this particular example they are isomorphic as abstract fields. The easiest way to see it is to check the genus of the corresponding curves which turns out to be zero and therefore $K \simeq L \simeq F$. To check if these two fields indeed split equivalently one can apply the classical Kummer-Dedekind Theorem which asserts that for all except finitely many ideals $\mathcal{P} \subset \mathbb F_q[T]$ the primes in $\mathcal{O}_K$ lying above $\mathcal P$ are in bijection with the factors of $f(y)~\mod~\mathcal{P}$. The code below checks that for a few random prime ideals in $\mathbb F_q[T]$ the factors of $f(y) \mod \mathcal P$ and $g(y) \mod \mathcal P$ coincide.

\begin{verbatim}
for i in [1..10] do
  P := RandomIrreduciblePolynomial(Fq, i);
  R<T> := ExtensionField<k, T | P>;
  RR<y> := PolynomialRing(R);
  RR;
  f := y^8 + T*y^2 + T;
  g := y^8 + T^2*y^6 + (2*T^6 + T^4 + T^3)*y^5 + (2*T^5 + 2*T^3)*y^4 + 
         (2*T^9 + 2*T^6 + 2*T^5)*y^3 +
         (T^16 + 2*T^13 + T^10 + 2*T^9 + T^8 + 2*T^6 + T^5)*y^2 + 
         (T^17 + 2*T^13 + T^10 + 2*T^9 + 2*T^8 + 2*T^7 + T^6)*y + 
         T^21 + T^15 + T^14 + T^13 + 2*T^12 + T^11 + 2*T^10 + 2*T^9 + 2*T^7 + T^6;
  "Factorization of f mod", P, Factorization(f); 
  "Factorization of g mod", P, Factorization(g); 
end for;
\end{verbatim} 

By scanning through the output one can find that for instance modulo $\mathcal{P} = T + 1$ the above polynomials stay irreducible: 
\begin{equation} 
\begin{split}
\bar{f}(y) =  \ & y^8 + 2y^2 + 2 \mod{\mathcal{P}}  \\
\bar{g}(y) = \ & y^8 + y^6 + 2y^5 + 2y^4 + y^3 + 1 \mod{\mathcal{P}}.
\nonumber
\end{split}
\end{equation}

Meanwhile modulo $\mathcal{P} = T^2 +1$ each polynomial factors as product of two irreducible polynomials of degrees two and six: 
\begin{equation} 
\begin{split}
\ \ \ \ \ \ \ \ \ \ \bar{f}(y) =  \ &(y^2 + T + 1)(y^6 + (2T + 2)y^4 + 2Ty^2 + 2T + 2) \mod{\mathcal{P}}  \\
\ \ \ \ \ \ \ \ \ \ \bar{g}(y) = \ &(y^2 + (2T + 1)y + 2T + 2) \times \\
                   \ & \times (y^6 + (T + 2)y^5 + 2Ty^4 + y^3 + (2T + 2)y + 2T + 1)~\mod{\mathcal{P}}.
\nonumber
\end{split}
\end{equation}

\subsection{Example with the Galois Group $\GL_2(\mathbb F_4)$ over $\mathbb F_2(T)$}
To construct an extension of $F =\mathbb F_2(T)$ with the Galois group $\GL_2(\mathbb F_4)$ we pick the Drinfeld module associated to the map $T \to \tau^2 + \tau + T$ and adjoint to $F$ the $(T^2 + T + 1)$~-~torsion points. The corresponding action of $(T^2 + T + 1)$ on $\mathbb F_q(T)$ is given by the following operator:
\begin{equation}
\begin{split}
(\tau^2 \ + \ & \tau + T)^2  + (\tau^2 + \tau + T) + 1 =  \\
& = \tau^4 + \tau^3 + T^4\tau^2 + \tau^3 + \tau^2 + T^2\tau + T\tau^2 + T\tau + T^2  + \tau^2 + \tau + T + 1 = \\
& = \tau^4 + (T^4 + T)\tau^2 + (T^2 + T + 1)\tau + T^2 + T + 1.
\nonumber
\end{split}
\end{equation}
And therefore the $(T^2 + T +1)$-torsion points are defined by the equation given below:
\begin{equation}\label{Eq3}
y^{15} + (T^4 + T)y^3 + (T^2+T+1)y + T^2 + T + 1.
\end{equation}

As before we verify with Magma that the Galois group of the Normal closure of the field generated by the roots of \ref{Eq3} is isomorphic to $\GL_2(\mathbb F_4)$. This produces explicit equations of the arithmetical equivalent pair of degree 15:
\begin{equation}
\begin{split}
f(y&)  \  = \  y^{15} + (T^4 + T)y^3 + (T^2 + T + 1)y + T^2 + T + 1 ; \\    
g(y&)  \ = \  y^{15} + (T^4 + T^2 + 1)y^{13}  + (T^{10} + T^7 + T^6 + T^5 + T^4 + T^3 + 1)y^{12} + \\ 
            + & \ (T^{12} + T^{11} + T^8 + T^7 + T^4 + T^3)y^{11} + (T^{16} + T^{14} + T^{10} + T^9 + T^8 + T^6 + T^5 + T + 1)y^{10} + \\
            + & \ (T^{20} + T^{18} + T^{16} + T^{14} + T^8 + T^6 + T^4 + T^2)y^9 + \\ 
            + & \ (T^{21} + T^{19} + T^{16} + T^{14} + T^{7} + T^5 + T^4 + T^2 + T)y^8 + \\
            + & \ (T^{26} + T^{23} + T^{16} + T^{15} + T^{14} + T^{13} + T^{12} + T^{10} + T^8 + T^7 + T^6 + T^2 + T)y^7 + \\ 
            + & \ (T^{30} + T^{28} + T^{27} + T^{25} + T^{23} + T^{22} + T^{21} + T^{20} + T^{19} + T^{18} + T^{15} + T^{14} + \\
            + & \ T^6 + T^5 + T^3 + 1)y^6 + \\
            + & \ (T^{29} + T^{28} + T^{25} + T^{24} + T^{23} + T^{22} + T^{21} + T^{20} + T^{18} + T^{17} + T^{15} + T^{14} + \\
            + & \ T^{13} + T^{12} + T^{11} + T^5 + T^3 + T + 1)y^5 + \\
            + & \ (T^{33} + T^{32} + T^{30} + T^{29} + T^{28} + T^{27} + T^{25} + T^{23} + T^{21} + T^{19} + T^{14} + \\
            + & \ T^{11} + T^10 + T^7 + T^5 + T^3 + T + 1)y^4 + \\
            + &\ (T^{36} + T^{33} + T^{31} + T^{30} + T^{26} + T^{23} + T^{18} + T^{14} + T^{12} + T^{11} + T^8 + T^7 + T^5 + T^4 + T + 1)y^3 + \\
            + &\ (T^{38} + T^{36} + T^{33} + T^{32} + T^{31} + T^{30} + T^{29} + T^{26} + T^{25} + T^{21} + \\
            + & \ T^{19} + T^{17} + T^{12} + T^{10} + T^4 + T^2 + T + 1)y^2 + \\ 
            + &(T^{40} + T^{38} + T^{37} + T^{36} + T^{33} + T^{32} + T^{31} + T^{29} + T^{25} + T^{20} + T^{17} + T^{14} + T^{13} + T^7 + \\
            + &\ T^5 + T^4 + 1)y + T^{44} + T^{43} + T^{42} + T^{40} + T^{38} + T^{37} + T^{34} + T^{33} + T^{30} + T^{21} + T^{20} + T^{19} + \\
            + &\ T^{18} + T^{17} + T^{16} + T^{13} + T^{11} + T^8 + T^7 + T^5 + T^4 + T^3 + 1.
\nonumber
\end{split}
\end{equation}

\section{A final remark}
A careful reader may observe that equations that we take to realise $\GL_r(\mathbb F_{q^m})$ do occur in the list of provided arithmetically equivalent pairs. This is not a coincidence and can be formulated as the following:
\begin{theorem} 
Let $\rho$ be a Drinfeld module of generic characteristic over $K = \mathbb F_q(T)$ without complex multiplication. Suppose that the rank $r$ of $\rho$ is at least $2$. Then for almost all irreducible $a \in \mathbb F_q[T] = A$ the isomorphism class of the extension $M$ generated by adjoining the elements of $\rho[a]$ to $K$ is not determined by its splitting type.
\end{theorem}
\begin{proof}
Indeed, as stated in the preliminary section $\rho[a] \simeq (A/aA)^{r}$ which is a vector space if $a$ is irreducible. Since $\rho$ does not have complex multiplication we know that for almost all $a$ the associated Galois group $G$ of the normal closure of $K(\rho[a])$ will be $\GL_r(\mathbb F_{q^m})$, where $m$ is the degree of $a$. Then example \ref{Example2} with $S = \{1\}$ guarantees that the stabilizer $H$ of $\rho[a] - \{ 0 \}$ forms a Gassmann triple $(G, H, H')$ for some $H' \subset G$. Since $A$ has only finitely many elements of a fixed degree we know that $q^m - 1 \ge 2$ for almost all $a$ and therefore the triple $(G, H, H')$ is non-trivial for almost all $a$. But this means that there exists   a field $L$ which is not isomorphic to $K(\rho[a])$ and that share the same splitting type.
\end{proof}

\newpage

\bibliography{mybib}{}
\bibliographystyle{plain}

\newpage

\tableofcontents

\end{document}